\documentclass[12pt,reqno,a4paper]{amsart}
\usepackage[utf8]{inputenc}
\usepackage[doi=true,isbn=false,style=alphabetic,sorting=nyt,backend=biber,maxnames=5,maxalphanames=5,giveninits=true]{biblatex}
\AtEveryBibitem{\clearlist{language}}
\bibliography{main}

\usepackage[breaklinks,colorlinks,plainpages,hypertexnames=false,plainpages=false]{hyperref}
\hypersetup{urlcolor=blue, citecolor=blue, linkcolor=blue}

\usepackage{amsmath}
\usepackage{amsthm}
\usepackage{thmtools}
\usepackage{amssymb}
\usepackage{amsfonts}
\usepackage{enumerate}
\usepackage{mathtools}
\usepackage{tikz-cd}
\usepackage{hhline}
\usepackage{stmaryrd}
\usepackage{comment}
\usepackage{enumitem}
\usepackage{centernot}
\usepackage[textwidth=25mm]{todonotes}
\usepackage{mathrsfs}
\usepackage{comment}
\usepackage{jlcode}
\usepackage{multicol}
\usepackage[capitalise, noabbrev]{cleveref} 
\usepackage[noend]{algorithmic} 

\usepackage{mathdots}
\usepackage{cleveref} 
\usepackage{nicematrix} 
\usepackage{mathalpha} 
\DeclareMathAlphabet{\dutchcal}{U}{dutchcal}{m}{n}

\usepackage[indent,skip=.2\baselineskip]{parskip}

\usepackage{array}
\newcolumntype{L}[1]{>{\raggedright\let\newline\\\arraybackslash\hspace{0pt}}m{#1}}
\newcolumntype{C}[1]{>{\centering\let\newline\\\arraybackslash\hspace{0pt}}m{#1}}
\newcolumntype{R}[1]{>{\raggedleft\let\newline\\\arraybackslash\hspace{0pt}}m{#1}}

\newcommand{\suchthat}{\;\ifnum\currentgrouptype=16 \middle\fi|\;}
\newcommand{\bigmid}{\left.\vphantom{\Big\{} \suchthat \vphantom{\Big\}}\right.}

\usepackage{tikz}
\usetikzlibrary{arrows}
\usetikzlibrary{decorations.pathreplacing}
\usetikzlibrary{matrix}
\usetikzlibrary{calc}
\usetikzlibrary{shapes}
\usetikzlibrary{patterns}
\usetikzlibrary{fit,backgrounds,scopes}

\newtheoremstyle{theoremstyle}
{10pt}      
{5pt}       
{\itshape}  
{}          
{\bfseries} 
{}         
{ }      
{}          

\newtheoremstyle{algorithmstyle}
{10pt}      
{5pt}       
{}  
{}          
{\bfseries} 
{}         
{ }      
{}          

\newtheoremstyle{examplestyle}
{10pt}      
{5pt}       
{}          
{}          
{\bfseries} 
{}         
{ }      
{}          

\makeatletter 
\newcommand{\subalign}[1]{%
  \vcenter{%
    \Let@ \restore@math@cr \default@tag
    \baselineskip\fontdimen10 \scriptfont\tw@
    \advance\baselineskip\fontdimen12 \scriptfont\tw@
    \lineskip\thr@@\fontdimen8 \scriptfont\thr@@
    \lineskiplimit\lineskip
    \ialign{\hfil$\m@th\scriptstyle##$&$\m@th\scriptstyle{}##$\hfil\crcr
      #1\crcr
    }%
  }%
}
\makeatother

\theoremstyle{theoremstyle}
\newtheorem{theorem}{Theorem}[section]
\newtheorem{lemma}[theorem]{Lemma}
\newtheorem{proposition}[theorem]{Proposition}

\theoremstyle{examplestyle}
\newtheorem{example}[theorem]{Example}
\newtheorem{definition}[theorem]{Definition}

\newtheorem{remark}[theorem]{Remark}

\newtheorem{convention}[theorem]{Convention}

\theoremstyle{algorithmstyle}
\newtheorem{algorithm}[theorem]{Algorithm}

\newcommand{\CC}{\mathbb{C}}
\newcommand{\RR}{\mathbb{R}}
\newcommand{\QQ}{\mathbb{Q}}
\newcommand{\ZZ}{\mathbb{Z}}

\DeclareMathOperator{\conv}{conv}

\DeclareMathOperator{\initial}{in}

\DeclareMathOperator{\Trop}{Trop}

\newcommand{\Lazypic}[2]{\begin{minipage}{#1} \vspace{0.1cm} \centering {#2}\vspace{0.1cm}\end{minipage}}

\DeclareMathOperator{\val}{val}

\begin{document}

\title{Zero-dimensional tropicalizations in \textsc{Oscar}}
\author{Arman Marti-Shahandeh, Yue Ren, Victoria Schleis}
\date{\today}
\begin{abstract}
  We present algorithms for computing zero-dimensional tropical varieties as implemented in \texttt{OscarZerodimensionalTropicalization.jl}.
  The algorithms include a workaround for a common practical issue arising when working with polynomials over inexact fields in existing software systems.
\end{abstract}
\maketitle

\section{Introduction}
Tropical varieties are a combinatorial shadow of algebraic varieties. They play a central role in tropical geometry, and, more broadly, in combinatorial algebraic geometry.  To understand and apply tropical geometry, efficient algorithms for the computation of tropical varieties are essential.  For example, showing that the tropicalization equals $\{0\}$ proves the finiteness of certain central configurations in celestial mechanics \cite{HamptonMoeckel2006,HamptonJensen2011}, and computing tropicalizations enables better ways to solve certain polynomial systems \cite{HelminckHenrikssonRen2024}.

Fundamentally, given an ideal $I\subseteq K[x_1,\dots,x_n]$ and a (possibly trivial) valuation $\val\colon K^\ast \rightarrow \RR$, the computational task is to find a set of polyhedra covering the \emph{tropical variety} $\Trop(I)\subseteq\RR^n$.  This is usually done via a traversal over $\Trop(I)$ \cite{BJSST2007,MarkwigRen2020}, and most non-Gr\"obner computations in that traversal can be boiled down to computations of zero-dimensional tropical varieties over function fields \cite{HofmannRen2018}. More specifically, we are interested in determining the tropical variety where $K$ is the algebraic closure of a rational function field $\mathfrak K(t)$ with its $t$-adic valuation, $I$ is zero-dimensional and $\Trop(I)$ consists of finitely many points.

In this paper, we report on a new implementation of the algorithm for computing zero-dimensional tropical varieties in \textsc{Oscar} \cite{Oscar}. Our algorithm works via triangular decomposition and root approximation as in \cite{HofmannRen2018}, but has several technical 
improvements over existing implementations.  

The main challenge we address is \textsc{Oscar}'s handling of polynomials over inexact fields, which is similar to that of \textsc{Macaulay2} and \textsc{Sagemath}, and different to that of \textsc{magma}.  Simply speaking, \textsc{Macaulay2}, \textsc{Oscar}, and \textsc{Sagemath} work with sparse polynomials, meaning that numerically zero terms are deleted, see \cref{sec:technicalProblems}.  This may destroy important precision information. We  solve this issue by introducing additional variables $u_1,\dots,u_n$ to track the error terms in \cref{sec:approximateRootsAndRootApproximations}.  
 
  Unlike \cite{HofmannRen2018}, we do not outsource the root approximation to \textsc{magma} \cite{magma}, which gives us much finer control on the precision required for computing $\Trop(I)$.  And unlike  a similar algorithm in \cite{JensenMarkwigMarkwig2008}, we do not track field extensions manually and instead make generous use of the \texttt{algebraic\_closure} functionality in \textsc{Oscar}, which improves the readability of our code significantly.

\subsection*{Funding}  Y.R. and V.S. are funded by the UKRI FLF \emph{Computational Tropical Geometry and its Applications} (MR/Y003888/1). V.S. was a member of the Institute for Advanced Study, funded by the Charles Simonyi Endowment.

\section{Background}
In this section, we fix our notation and review basics of valued fields, tropical geometry, and triangular decompositions of Gr\"obner bases. 

\begin{convention}
  \label{con:coefficientField}
  In this paper, let $\mathfrak K = \mathfrak K^{\mathrm{al}}$ be an algebraically closed field.  Let $\mathfrak K(\!(t)\!)$ denote its field of Laurent series, and let $K\coloneqq \mathfrak K(\!(t)\!)^{\mathrm{al}}$ be its algebraic closure.  If $\mathfrak K$ has characteristic $0$, then $K=\mathfrak K\{\!\{t\}\!\}$ is the field of Puiseux series \cite[Theorem 2.1.5]{MaclaganSturmfels}.  If $\mathfrak K$ has positive characteristic, $K$ is a subfield of the field of Hahn series \cite[Section 2]{Hahn1907} containing the field of Puiseux series:
  \begin{equation*}
    \mathfrak K\{\!\{t\}\!\}\subseteq K\subseteq \mathfrak K\llbracket t^\QQ\rrbracket \coloneqq \Big\{ \sum_{w\in S} c_w\cdot t^w \bigmid S\subseteq \QQ \text{ well ordered and } c_w\in\mathfrak K\Big\}.
  \end{equation*}
  Recall that $S\subseteq\QQ$ being well-ordered means any subset of $S$ has a minimum. This allows us to view elements of $K$ as generalized power series over $\mathfrak K$.

  The field $K$ has a natural valuation under which $\mathfrak K$ becomes its residue field:
  \begin{equation*}
    \val\colon K^\ast\coloneqq K\setminus\{0\}\longrightarrow\RR,\quad \sum_{w\in S} c_w\cdot t^w\longmapsto \min\{w\in S\mid c_w\neq 0\}.
  \end{equation*}
  When we want to focus on the lowest coefficients of $z\in K$, we will often write it as
  \begin{equation}
    \label{eq:rootApproximationForm}
    z=c_0\cdot t^{w_0}+c_1\cdot t^{w_1}+\dots+c_{r-1}\cdot t^{w_{r-1}}+u\cdot t^{w_{r}}.
  \end{equation}
  where $\val(z)=w_0<\dots <w_r$ and $c_0,\dots,c_{r-1}\in \mathfrak K^\ast$ and $u\in K^\ast$ with $\val(u)=0$.

  We further fix a multivariate Laurent polynomial ring $K[x^\pm]\coloneqq K[x_1^\pm,\dots,x_n^\pm]$ and its polynomial subring $K[x]\coloneqq K[x_1,\dots,x_n]$.
\end{convention}

\subsection{Tropical Varieties}\label{sec:tropicalVarieties}
In this section, we recall the definition of tropical varieties.  As we will rely on techniques adjacent to classical Gr\"obner basis theory, see \cref{sec:triangularDecomposition}, we will work with polynomial ideals rather than Laurent polynomial ideals as in \cite{MaclaganSturmfels}.  For the sake of simplicity, we use one of the equivalent definitions in the Fundamental Theorem of Tropical Geometry \cite[Theorem 3.2.3]{MaclaganSturmfels} 

\begin{definition}
  \label{def:tropicalVariety}
  The \emph{tropical variety} of an ideal $I\subseteq K[x^\pm]$ is
  \begin{equation*}
    \Trop(I)\coloneqq \mathrm{cl}\Big(\val(V(I)\cap (K^\ast)^n)\Big)\subseteq\RR^n,
  \end{equation*}
  where $V(I)\subseteq (K^\ast)^n$ is the very affine variety of $I$, $\val(\cdot)$ denotes the coordinatewise valuation, and $\mathrm{cl}(\cdot)$ denotes the closure in the Euclidean topology.

  For principal ideals $\langle f\rangle\subseteq K[x^\pm]$ we abbreviate $\Trop(f)\coloneqq\Trop(\langle f\rangle)$. 
\end{definition}

Every ideal of Laurent polynomials $I \subseteq K[x^\pm]$ induces an ideal of polynomials $I\cap K[x]$. This allows us to interpret $I$ as an ideal of polynomials over $K[x]$, which we will do in the remainder  of this paper.

When $I$ is zero-dimensional, $\Trop(I)$ consists of a finitely many points in $\RR^n$. Tropicalizations of univariate polynomials are zero-dimensional and can be read off their Newton polygons (also referred to as Newton diagrams in \cite{MaclaganSturmfels} and as extended Newton polyhedra in \cite{Joswig}):

\begin{definition}
  \label{def:newtonPolygon}
  Let $f' \coloneqq \sum_{j\in S}a_j\cdot x_i^j\in K[x_i]$ be a univariate polynomial with finite support $S\subseteq\ZZ_{\geq 0}$. 
  Its \emph{Newton polygon} is
  \begin{equation*}
    \Delta(f')=\conv\Big(\big\{(j,\val(a_j))\mid j\in S\big\}\Big) + \RR_{\geq 0}\cdot (0,1) \subseteq\RR^2.
  \end{equation*}
  A \emph{lower edge} of $\Delta(f')$ is an edge connecting two vertices $(j_1,\val(a_j))$ and $(j_2,\val(a_{j_2}))$ for distinct $j_1,j_2\in S$, say $j_1<j_2$, and we refer to $(\val(a_{j_2})-\val(a_{j_1}))/(j_2-j_1)$ as a \emph{slope} of $\Delta(f')$.
\end{definition}

The slopes of edges of the Newton polygon determine the tropicalization.

\begin{lemma}
  \label{lem:newtonPolygon}
  Let $f'\in K[x_i]$ be a univariate polynomial in $x_i$.  Then
   \begin{equation*}
    \label{eq:newtonPolygon}
    \Trop(f')=\{-s\in\RR\mid s\text{ slope of } \Delta(f')\}.
  \end{equation*}
\end{lemma}

In the best case, \cref{lem:newtonPolygon} alone suffices to compute tropicalizations of zero-dimensional ideals, as seen in \cref{fig:seriesOfNewtonPolygons}.

\begin{example}[{\cite[Example 2.14]{HofmannRen2018}}]
  \label{ex:tropicalVariety}
  Let $I=\langle f_1, f_2, f_3\rangle \subseteq \CC\{\!\{t\}\!\}[x_1,x_2,x_3]$ be generated by
  \begin{align*}
    f_1 =  tx_1^2+x_1+1, \qquad f_2 = tx_1x_2^2+x_1x_2+1, \qquad f_3 = x_1x_2x_3+1.
  \end{align*}
  It is straightforward to see that $I$ is zero-dimensional and how one can compute $\Trop(I)$ via back-substitution:

  According to $\Delta(f_1)$ we have $\Trop(f_1)=\{0,-1\}$.  Letting $z_1\in V(f_1)$ be the root with $\val(z_1)=0$ and $f_2(z_1,x_2)\in K[x_2]$ be the substituted polynomial, we see on $\Delta(f_2(z_1,x_2))$ that $\Trop(f_2(z_1,x_2))=\{0,-1\}$.  Exhausting this process yields the tree of Newton polygons in \cref{fig:seriesOfNewtonPolygons} and
  \[ \Trop(I) = \{ (0, 0, 0), (0, -1, 1), (-1, 1, 0), (-1, -1, 2)\}. \]
\end{example}
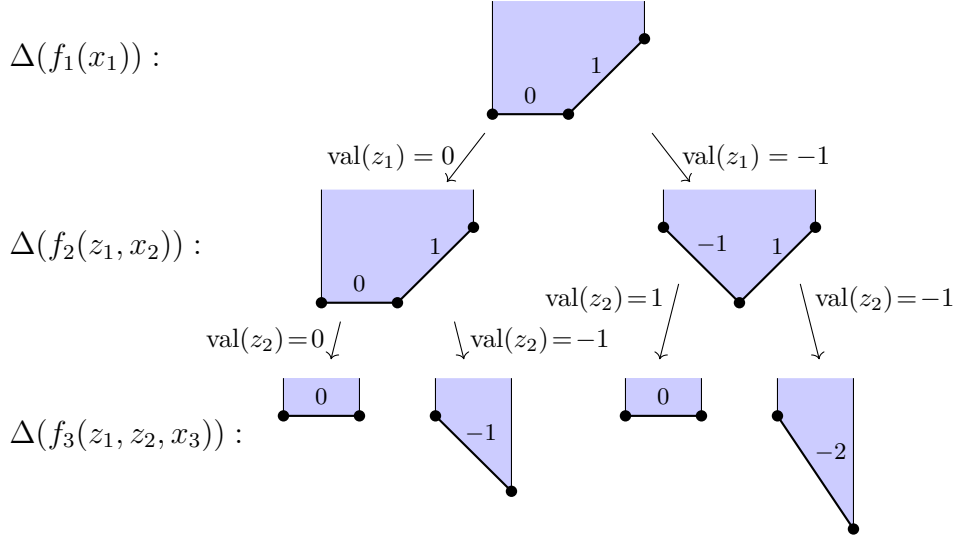
\begin{figure}[t]
  \flushleft
  \begin{tikzpicture}
    \node[anchor=west] at (-7.5,0.75) {$\Delta(f_1(x_1)):$};
    \fill[blue!20] (-1,1.5) -- (-1,0) -- (0,0) -- (1,1) -- (1,1.5) -- cycle;
    \draw (-1,1.5) -- (-1,0)
    (1,1) -- (1,1.5);
    \draw[thick,black]
    (-1,0) -- node[above,black,font=\scriptsize] {$0$} (0,0);
    \draw[thick,black]
    (0,0) -- node[anchor=south east,xshift=0.1cm,yshift=-0.1cm,black,font=\scriptsize] {$1$} (1,1);
    \fill (-1,0) circle (0.075cm);
    \fill (0,0) circle (0.075cm);
    \fill (1,1) circle (0.075cm);

    \draw[->,black] (-1.1,-0.25) -- node[left,black,font=\footnotesize] {$\val(z_1)=0$} ($(-1.1,-0.25)+(-0.5,-0.65)$);
    \draw[->,black] (1.1,-0.25) -- node[right,black,font=\footnotesize] {$\val(z_1)=-1$} ($(1.1,-0.25)+(0.5,-0.65)$);

    \node[anchor=west] at (-7.5,-1.75) {$\Delta(f_2(z_1,x_2)):$};
    \node (o11) at (-2.25,-2.5) {};
    \fill[blue!20] ($(-1,1.5)+(o11)$) -- ($(-1,0)+(o11)$) -- ($(0,0)+(o11)$) -- ($(1,1)+(o11)$) -- ($(1,1.5)+(o11)$) -- cycle;
    \draw ($(-1,1.5)+(o11)$) -- ($(-1,0)+(o11)$)
    ($(1,1)+(o11)$) -- ($(1,1.5)+(o11)$);
    \draw[thick,black]
    ($(-1,0)+(o11)$) -- node[above,black,font=\scriptsize] {$0$} ($(0,0)+(o11)$);
    \draw[thick,black]
    ($(0,0)+(o11)$) -- node[above,black,font=\scriptsize] {$1$} ($(1,1)+(o11)$);
    \fill ($(-1,0)+(o11)$) circle (0.075cm);
    \fill ($(0,0)+(o11)$) circle (0.075cm);
    \fill ($(1,1)+(o11)$) circle (0.075cm);

    \draw[->,black] ($(-0.75,-0.25)+(o11)$) -- node[left,black,font=\footnotesize] {$\val(z_2)\!=\!0$} ++(-0.125,-0.5);
    \draw[->,black] ($(0.75,-0.25)+(o11)$) -- node[right,black,font=\footnotesize] {$\val(z_2)\!=\!-1$} ++(0.125,-0.5);

    \node (o12) at (2.25,-1.5) {};
    \fill[blue!20] ($(-1,0.5)+(o12)$) -- ($(-1,0)+(o12)$) -- ($(0,-1) + (o12)$) -- ($(1,0)+(o12)$) -- ($(1,0.5)+(o12)$) -- cycle;
    \draw ($(-1,0.5)+(o12)$) -- ($(-1,0)+(o12)$)
    ($(1,0)+(o12)$) -- ($(1,0.5)+(o12)$);
    \draw[thick,black]
    ($(-1,0)+(o12)$) -- node[above,black,font=\scriptsize,xshift=4pt] {$-1$} ($(0,-1)+(o12)$) -- node[above,black,font=\scriptsize] {$1$} ($(1,0)+(o12)$);
    \fill ($(-1,0)+(o12)$) circle (0.075cm);
    \fill ($(0,-1)+(o12)$) circle (0.075cm);
    \fill ($(1,0)+(o12)$) circle (0.075cm);

    \draw[->,black] ($(-0.8,-0.75)+(o12)$) --  node[pos=0.2,left,black,font=\footnotesize] {$\val(z_2)\!=\!1$} ++(-0.25,-1);
    \draw[->,black] ($(0.8,-0.75)+(o12)$) --  node[pos=0.2,right,black,font=\footnotesize] {$\val(z_2)\!=\!-1$} ++(0.25,-1);

    \node[anchor=west] at (-7.5,-4.25) {$\Delta(f_3(z_1,z_2,x_3)):$};
    \node (o21) at ($(o11)+(-1.5,-1.5)$) {};
    \fill[blue!20] ($(0,0.5)+(o21)$) -- ($(0,0)+(o21)$) -- ($(1,0)+(o21)$) -- ($(1,0.5)+(o21)$) -- cycle;
    \draw ($(0,0.5)+(o21)$) -- ($(0,0)+(o21)$)
    ($(1,0)+(o21)$) -- ($(1,0.5)+(o21)$);
    \draw[thick,black]
    ($(0,0)+(o21)$) -- node[above,black,font=\scriptsize] {$0$} ($(1,0)+(o21)$);
    \fill ($(0,0)+(o21)$) circle (0.075cm);
    \fill ($(1,0)+(o21)$) circle (0.075cm);

    \node (o22) at ($(o11)+(0.5,-1.5)$) {};
    \fill[blue!20] ($(0,0.5)+(o22)$) -- ($(0,0)+(o22)$) -- ($(1,-1)+(o22)$) -- ($(1,0.5)+(o22)$) -- cycle;
    \draw ($(0,0.5)+(o22)$) -- ($(0,0)+(o22)$)
    ($(1,-1)+(o22)$) -- ($(1,0.5)+(o22)$);
    \draw[thick,black]
    ($(0,0)+(o22)$) -- node[above,black,font=\scriptsize,xshift=0.1cm] {$-1$} ($(1,-1)+(o22)$);
    \fill ($(0,0)+(o22)$) circle (0.075cm);
    \fill ($(1,-1)+(o22)$) circle (0.075cm);

    \node (o23) at ($(o12)+(-1.5,-2.5)$) {};
    \fill[blue!20] ($(0,0.5)+(o23)$) -- ($(0,0)+(o23)$) -- ($(1,0)+(o23)$) -- ($(1,0.5)+(o23)$) -- cycle;
    \draw ($(0,0.5)+(o23)$) -- ($(0,0)+(o23)$)
    ($(1,0)+(o23)$) -- ($(1,0.5)+(o23)$);
    \draw[thick,black]
    ($(0,0)+(o23)$) -- node[above,black,font=\scriptsize] {$0$} ($(1,0)+(o23)$);
    \fill ($(0,0)+(o23)$) circle (0.075cm);
    \fill ($(1,0)+(o23)$) circle (0.075cm);

    \node (o24) at ($(o12)+(0.5,-2.5)$) {};
    \fill[blue!20] ($(0,0.5)+(o24)$) -- ($(0,0)+(o24)$) -- ($(1,-1.5)+(o24)$) -- ($(1,0.5)+(o24)$) -- cycle;
    \draw ($(0,0.5)+(o24)$) -- ($(0,0)+(o24)$)
    ($(1,-1.5)+(o24)$) -- ($(1,0.5)+(o24)$);
    \draw[thick,black]
    ($(0,0)+(o24)$) -- node[above,black,xshift=0.2cm,font=\scriptsize] {$-2$} ($(1,-1.5)+(o24)$);
    \fill ($(0,0)+(o24)$) circle (0.075cm);
    \fill ($(1,-1.5)+(o24)$) circle (0.075cm);
  \end{tikzpicture}\vspace{-3mm}
  \caption{A possible series of Newton polygons of $F$. }
  \label{fig:seriesOfNewtonPolygons}
\end{figure}

While it sufficed to know the valuation of the roots in \cref{ex:tropicalVariety}, one often needs to approximate as in \cref{eq:rootApproximationForm}.  The lowest coefficients of the roots are determined by initial forms.

\begin{definition}
  \label{def:initialForm}
  For $a\in K$ with $\val(a)\geq 0$ we write $\overline{a}\in\mathfrak K$ for the residue of $a$ in $\mathfrak K$ and for any $a\in K$ we write $\initial(a)\coloneqq \overline{t^{-\val(a)}\cdot a}\in \mathfrak K$.  Note that $\initial(a)$ is the lowest non-zero coefficient of $a$ regarded as a power series over $\mathfrak K$.

The \emph{initial form} of a polynomial $f\coloneqq \sum_{\alpha\in S}a_\alpha\cdot x^\alpha \in K[x]$ with finite support $S\subseteq\ZZ_{\geq 0}^n$ and coefficients $a_\alpha\in K^\ast$ with respect to a weight vector $w\in\RR^n$ is
  \begin{equation*}
    \initial_w(f)\coloneqq \sum_{\substack{w\cdot \alpha+\val(a_\alpha)\\\text{minimal}}} \initial(a_\alpha)\cdot x^\alpha\in \mathfrak K[x].
  \end{equation*}
\end{definition}

\begin{lemma}
  \label{lem:initialForm}
  Let $f'\in K[x_i]$ be a univariate polynomial in $x_i$.  Then for $w\in\Trop(f')$
  \begin{equation*}
    \label{eq:initialForm}
    V(\initial_w(f'))=\{ \initial(z)\in\mathfrak K\mid z\in V(f') \text{ and }\val(z)=w \}.
  \end{equation*}
\end{lemma}
\begin{proof}
  Contained in the proof of \cite[Proposition 3.1.5]{MaclaganSturmfels}.
\end{proof}

\subsection{Triangular decompositions}\label{sec:triangularDecomposition}

\begin{definition}
  \label{eq:triangularSet}
  A set of polynomials $F\subseteq K[x_1,\dots,x_n]$ is a \emph{triangular set} if it is of the form $F=\{f_1, \dots, f_n\}$ with $f_i \in K[x_1,\dots,x_i]\setminus K[x_1,\dots,x_{i-1}]$, i.e., the $i$-th polynomial only contains the first $i$ variables.
\end{definition}

\begin{proposition}[{\cite[Corollary 4.7.4]{GreuelPfister2007}}]
  \label{prop:triangularDecomposition}
  Let $I\subseteq K[x]$ be a zero-dimensional ideal.  Then there are triangular sets $F_1,\dots,F_k\subseteq K[x]$ such that
  \begin{equation*}
    \sqrt{I} = \sqrt{\langle F_1\rangle}\cap\dots\cap\sqrt{\langle F_k\rangle},
  \end{equation*}
  where $\sqrt{\cdot}$ denotes the radical of the ideal.
\end{proposition}

\begin{remark}
  \label{rem:triangularDecomposition}
    A triangular decomposition of an ideal $I$ can be obtained by computing a
    Gr\"obner basis of $I$ with respect to the lexicographical order. Given a Gr\"obner basis $G$ of $I$ for another order, the \texttt{FGLM}-algorithm \cite{faugereGianniLazardMora} efficiently converts $G$ to a lexicographical Gr\"obner basis.
\end{remark}

\section{Technical limitations in existing mathematical software}\label{sec:technicalProblems}
In this section, we briefly discuss technical limitations in existing mathematical software, which necessitate the concepts in \cref{sec:approximateRootsAndRootApproximations}.  In essence, there are two ways to work with Puiseux series:
\begin{enumerate}
\item \label{enumitem:inexact} inexactly up to finite precision (relative or absolute).  Examples include \texttt{PuiseuxSeriesRing} in \textsc{Magma} \cite{magma}, \texttt{PuiseuxSeriesField} in \textsc{Oscar} \cite{Oscar,OscarBook}, and \texttt{PuiseuxSeriesRing} in \textsc{SageMath} \cite{sagemath}.
\item \label{enumitem:exact} exactly but only allowing finite Puiseux series.  Examples include \textsc{PuiseuxPolynomial.jl} \cite{PuiseuxPolynomials.jl}, a standalone package for working with Puiseux polynomials in \textsc{julia}, and \textsc{OscarPuiseuxPolynomial.jl} \cite{OscarPuiseuxPolynomials.jl}, a package for working with Puiseux polynomials in \textsc{Oscar}.
\end{enumerate}
And while \textsc{Macaulay2} \cite{M2} does not support Puiseux series, its fields \texttt{RR} and \texttt{CC} operate similar to Option \eqref{enumitem:inexact}.

Both options can be problematic for the task at hand in their own ways.

In Option \eqref{enumitem:inexact}, when computing with polynomials over inexact fields, coefficients may become numerically $0$ but with low precision.   For example, the constant coefficient of $g\coloneqq f-(1+O(t))$ for $f\coloneqq (1+O(t^9))\cdot x+(1+O(t^9))\cdot y+(1+O(t^9))$ is numerically zero but only with low precision. A computer algebra system can deal with this in two ways: either by continuing to track the constant term even though it is numerically zero and remembering its low precision, or by deleting the constant term.
\textsc{Magma} does the former, 
while \textsc{Oscar}, \textsc{Macaulay2} and \textsc{Sagemath} do the latter.  

The latter can be problematic:
For example, \cref{fig:inexactCoefficientRingProblems} shows $g$ having no constant term in \textsc{Oscar} and thus evaluating it at $(1,1)$, yielding an incorrect high precision.

In Option \eqref{enumitem:exact}, the problem is evident:  It is difficult to approximate the root of say $f=(1-t^2)\cdot x - 1\in \CC\{\!\{t\}\!\}[x]$, i.e., approximate $z=1+t^2+t^4+\dots\in V(f)$, if the software does not distinguish between $1$ as an approximation to $z$ up to $t^0$ and $1+0\cdot t$ as an approximation to $z$ up to $t^1$.

\begin{figure}[t]
  \centering
  \begin{jllisting}
julia> K,t = puiseux_series_field(QQ, 9, "t"); # setting relative precision to 9
julia> R,(x,y) = K["x","y"];
julia> f = x + y + 1 # constant term only given with relative precision 1
x + y + 1
julia> g = f - (1+O(t)) # constant term is deleted as coefficient numerically 0
x + y
julia> evaluate(g, [1,1]) # evaluation wrongfully claims relative precision 9
2 + O(t^9)
  \end{jllisting}\vspace{-4mm}
  \caption{\textsc{Oscar} removing monomials with coefficients that are numerically $0$.}
  \label{fig:inexactCoefficientRingProblems}
\end{figure}

To work around both problems, the next section introduces a framework to track the necessary precisions for our task by introducing $n$ new variables.

\section{Root approximations}\label{sec:approximateRootsAndRootApproximations}
In this section, we introduce a symbolic workaround to the technical problem outlined in \cref{sec:technicalProblems}.  We then discuss how to perform root approximations of univariate polynomials whilst using this workaround.  

\subsection{Approximate roots}\label{sec:approximateRoots}
To symbolically track numerical uncertainties  in our computation, we work over the following ring:

\begin{definition}
  \label{def:uncertaintyRing}
  The \emph{uncertainty ring} is the polynomial ring $K[u]\coloneqq K[u_1,\dots,u_n]$.  On $K[u]$ we define a valuation extending that of $K$:
  \begin{equation*}
    \val\colon K[u]^\ast\rightarrow \RR,\quad \sum_{\alpha\in\ZZ_{\geq0}^n}a_\alpha u^\alpha\mapsto \min\Big( \{\val(a_\alpha)\mid \alpha\in\ZZ_{\geq 0}^n \text{ with } a_\alpha\neq 0 \}\Big).
  \end{equation*}
 We use $\tilde f_i \in K[u][x_i]$ to denote univariate approximate polynomials. Further, any $\tilde f\in K[u][x]$ will always be expressed as $\tilde f \coloneqq \sum_{\alpha\in S}\tilde a_\alpha\cdot x^\alpha$ for some finite $S\subseteq\ZZ^n_{\geq 0}$ and non-zero $\tilde a_j\in K[u_1,\dots,u_{i-1}]^\ast$.
\end{definition}

Treating $K[u]$ as a valued ring, the definitions of initials, Newton polygons, and initial forms extend straightforwardly from \cref{con:coefficientField}, \cref{def:newtonPolygon} and \cref{def:initialForm} respectively.

The variables $u_1,\dots,u_n\in K[u]$ represent arbitrary elements of $K$ of valuation $0$.  We will use $u_1$ for the tails of the roots $z_1\in V(f_1)$, and the remaining $u_i$ for the tails of the roots $z_i\in V(f_i(z_1,\dots,z_{i-1},x_i))$. 

\begin{convention}
  \label{con:triangularSet}
  For the remainder of the section, let $I=\langle F\rangle\subseteq K[x]$ be an ideal generated by a triangular set $F=\{f_1,\dots,f_n\}$ with $f_i\in K[x_1,\dots,x_i]$.
\end{convention}

\begin{definition}
  \label{def:newtonPolygonWellDefined}
  Let $\tilde f\in K[u_1,\dots,u_{i-1}][x_i]$ be a univariate approximate polynomial.  We say the Newton polygon $\Delta(\tilde f)$ is \emph{well-defined} or \emph{unique}, if
  \begin{equation}
    \label{eq:newtonPolygonWellDefined}
    \Delta(\underbrace{\tilde f|_{u_\ell=z_\ell}}_{\in K[x_i]})=\Delta(\underbrace{\tilde f|_{u_\ell=z_\ell'}}_{\in K[x_i]})
  \end{equation}
  for all $z_1,\dots,z_{i-1},z_1',\dots,z_{i-1}'\in K^\ast$ with $\val(z_\ell)=\val(z_\ell')=0$.
  If $\Delta(\tilde f)$ is unique, we define $\Delta(\tilde f)$ to be the polygon in \cref{eq:newtonPolygonWellDefined} and $\Trop(\tilde f)$ to be the set of its negated slopes.
\end{definition}

\begin{lemma}
  \label{lem:newtonPolygonWellDefined}
  Let $\tilde f=\sum_{\alpha\in S}\tilde a_\alpha\cdot x^\alpha\in K[u_1,\dots,u_{i-1}][x_i]$ be a univariate approximate polynomial. Then $\Delta(\tilde f)$ is unique  if and only if
    $\tilde a_j\in K[u_1,\dots,u_{i-1}]$ is a monomial for all vertices $(j,\val(\tilde a_j))\in\Delta(\tilde f).
  $
\end{lemma}

\begin{proof}
  We start by showing ``$\Leftarrow$". Assume that $ \tilde a_j\in K[u_1,\dots,u_{i-1}]$ is a monomial for all vertices $(j,\val(\tilde a_j))\in\Delta(\tilde f)$. Then for all $z_1,\dots,z_{i-1},z_1',\dots,z_{i-1}'\in K^\ast$ with $\val(z_\ell)=\val(z_\ell')=0$ we have
  \begin{align*}
       \Delta(\tilde f|_{u_\ell=z_\ell}) &=  \conv\Big(\big\{(j,\val(\tilde a_j|_{u_\ell=z_\ell}))\mid j\in S\big\}\Big) + \RR_{\geq 0}\cdot (0,1) \\
        &=  \conv\Big(\big\{(j, \val(\tilde a_j) \mid j\in S\big\}\Big) + \RR_{\geq 0}\cdot (0,1) \\
        &=  \conv\Big(\big\{(j,\val(\tilde a_j|_{u_j=z_j'}))\mid j\in S\big\}\Big) + \RR_{\geq 0}\cdot (0,1)  = \Delta(\tilde f|_{u_j=z_j'}).
    \end{align*}
    This proves that $\Delta(\tilde f)$ is unique.

    We now show ``$\Rightarrow$". Assume there exists some vertex $(j,\val(\tilde a_j))\in\Delta(\tilde f)$ such that $\tilde a_j\in K[u_1,\dots,u_{i-1}]$ is not a monomial. Then there are $z_1,\dots,z_{i-1},z_1',\dots,z_{i-1}'\in K^\ast$ with $\val(z_\ell)=\val(z_\ell')=0$ such that $\tilde a|_{u_j=z_j}\neq 0$ with $\val(\tilde a|_{u_j=z_j})=\val(\tilde a_j)$ and $\tilde a|_{u_j=z_j'}=0$.  Consequently, $\Delta(\tilde f|_{u_j=z_j})$ contains vertex $(j,\val(\tilde a_j))$ and $\Delta(\tilde f|_{u_j=z_j'})$ does not.  This implies that $\Delta(\tilde f)$ is not unique.
\end{proof}

Next we introduce approximate roots and characterize them algebraically for adaptation in our implementation.

\begin{definition}
  \label{def:approximateRoots}
  Let $\tilde f_i\in K[u_1,\dots,u_{i-1}][x_i]$ be a univariate polynomial over the uncertainty ring with a unique Newton polygon.
  An \emph{approximate root} of $\tilde f_i$ is an element $\tilde z_i\in K[u_i]$ of the form
  \begin{equation*}
    \tilde z_i = c_1t^{w_1}+\dots+c_{r-1} t^{w_{r-1}} + u_i \cdot t^{w_r}
  \end{equation*}
  for some $r\geq 1$, $w_1<\dots<w_r$, $c_j\in \mathfrak K^\ast$ for $j<r$ such that for all $z_1,\dots,z_{i-1}\in K^\ast$ with $\val(z_j)=0$ there is an $z_i\in K^\ast$ with $\val(z_i)=0$ such that $\tilde f(\tilde z)|_{u_1=z_1,\dots,u_i=z_i}=0$.

  We say a set of approximate roots $\tilde V\subseteq K[u_i]$ is \emph{complete}, if for all $z_1,\dots,z_{i-1}\in K^\ast$ with $\val(z_\ell)=0$ and all $z\in V(\tilde f|_{u_1=z_1,\dots,u_{i-1}=z_{i-1}})\subseteq K$ there are $\tilde z_i\in\tilde V$ and $z_i\in K^\ast$ with $\val(z_i)=0$ such that $z=\tilde z_i|_{u_i=z_i}$.
\end{definition}

\begin{example}
  \label{ex:approximateRoots}
  Consider $\tilde f_2 = (x_2-t-t^2u_1)(x_2-1-t-t^2u_1)\in K[u_1][x_2]$. Two complete sets of approximate roots are $\widetilde V=\{u_2, tu_2\}$ and $\widetilde V'=\{1+t+t^2u_2, t+t^2u_2\}$.  In particular, this shows that complete sets of approximate roots are not unique. 
\end{example}

\begin{lemma}
  \label{lem:approximateRoots}
  Let $\tilde f_i\in K[u_1,\dots,u_{i-1}][x_i]$ have a unique Newton polygon.  Then
    $\tilde z_i \in K[u_1,\dots,u_i] \text{ is an approximate root of } \tilde f_i$ if and only if
    $$\initial(\tilde f_i(\tilde z_i))\in \mathfrak K[u_1,\dots,u_{i-1}][u_i]$$ is not a monomial in $u_i$.
\end{lemma}
\begin{proof}
  In the following, let $z_1,\dots,z_{i-1}\in K^\ast$ with $\val(z_\ell)=0$.

  To see ``$\Rightarrow$'', assume that $\initial(\tilde f_i(\tilde z_i))$ is a monomial in $u_i$.  Then there is no $\overline z_i\in\mathfrak K^\ast$ such that $\initial(\tilde f_i(\tilde z_i)|_{u_1=z_1,\dots,u_{i-1}=z_{i-1}})|_{u_i=\overline{z}_i}=0$, and thus there is also no $z_i\in K^\ast$ with $\val(z_i)=0$ such that $\initial(\tilde f_i(\tilde z_i)|_{u_1=z_1,\dots,u_{i}=z_{i}})=0$.  This contradicts that $\tilde z_i$ is an approximate root of $\tilde f_i$.

 To show ``$\Leftarrow$'', assume that $\initial(\tilde f_i(\tilde z_i))\in \mathfrak K[u_1,\dots,u_{i-1}][u_i]$ is not a monomial in $u_i$. Then there is some $\overline z_i\in\mathfrak K^\ast$ such that $\initial(\tilde f_i(\tilde z_i)|_{u_1=z_1,\dots,u_{i-1}=z_{i-1}})|_{u_i=\overline{z}_i}=0$, and thus there is also no $z_i\in K^\ast$ with $\val(z_i)=0$ such that $\initial(\tilde f_i(\tilde z_i)|_{u_1=z_1,\dots,u_{i}=z_{i}})=0$.  This shows that $\tilde z_i$ is an approximate root of $\tilde f_i$.
\end{proof}

In \cref{ex:approximateRoots}, we have seen that approximate roots are not unique and can have varying precision.  This motivates the following:  

\begin{definition}
  \label{def:approximateRootsMaximalPrecision}
  Let $\tilde f_i\in K[u_1,\dots,u_{i-1}][x_i]$ be a univariate polynomial over the uncertainty ring with a unique Newton polygon, and let $\tilde z_i\in K[u_i]$ be an approximate root of $\tilde f_i$.
  We say $\tilde z_i$ has \emph{maximal precision}, if $\tilde f_i(\tilde z_i)|_{u_i=x_i}=\sum_{j=0}^d \tilde a_j\cdot x_i^j \in K[u_1,\dots,u_{i-1}][x_i]$ has a unique Newton polygon $\Delta(\tilde f_i(\tilde z_i)|_{u_i=x_i})$ and there is a point $(j,\val(\tilde a_j))\in \Delta(\tilde f_i(\tilde z_i)|_{u_i=x_i})$ on a lower edge such that $\initial(\tilde a_j)\in \mathfrak K[u_1,\dots,u_{i-1}]\setminus \mathfrak K$.
\end{definition}

\begin{example}
  \label{ex:approximateRootsMaximalPrecision}
  In \cref{ex:approximateRoots}, the approximate roots of $\tilde f_2$ in $\widetilde V$ are intuitively of lower precision than those in $\widetilde V'$.  This matches the fact that they are not of maximal precision by \cref{def:approximateRootsMaximalPrecision}:
  \begin{itemize}
  \item $\tilde f_2(u_2)|_{u_2=x_2}=\tilde f_2 = x_2^2 + (-2t^2u_1 - 2t - 1)x_2 + t^4u_1^2 + (2t^3 + t^2)u_1 + t^2 + t$,
  \item $\tilde f_2(tu_2)|_{u_2=x_2}=t^2x_2^2 + (-2t^3u_1 - 2t^2 - t)x_2 + t^4u_1^2 + (2t^3 + t^2)u_1 + t^2 + t$,
  \end{itemize}
  and while the Newton polygons $\Delta(\tilde f_i(\tilde z_i)|_{u_i=x_i})$ are unique, all coefficient initials $\initial(\tilde a_j)$ are in $\mathfrak K$, see \cref{fig:approximateRootsMaximalPrecision}.

  In contrast, the approximate roots in $\widetilde V'$ are of maximal precision:
  \begin{itemize}
  \item $\tilde f_2(1+t+t^2u_2)|_{u_2=x_2}=t^4x_2^2 + (-2t^4u_1 + t^2)x_2 + t^4u_1^2 - t^2u_1$,
  \item $\tilde f_2(t+t^2u_2)|_{u_2=x_2}=t^4x_2^2 + (-2t^4u_1 - t^2)x_2 + t^4u_1^2 + t^2u_1$,
  \end{itemize}
  and there are coefficient initials $\initial(\tilde a_j)$ in $\mathfrak K[u_1]\setminus\mathfrak K$, see \cref{fig:approximateRootsMaximalPrecision}.
\end{example}

\begin{figure}[t]
  \centering
  \begin{tikzpicture}
    \node (z11) at (0,0)
    {
      \begin{tikzpicture}
        \coordinate (v0) at (0,1);
        \coordinate (v1) at (1,0);
        \coordinate (v2) at (2,0);
        \coordinate (w0) at (0,1.5);
        \coordinate (w2) at (2,1.5);
        \fill[blue!20,draw=black,thick] (w0) -- (v0) -- (v1) -- (v2) -- (w2);
        \fill (v0) circle (2pt);
        \fill (v1) circle (2pt);
        \fill (v2) circle (2pt);
        \node[below,font=\footnotesize] at (v0) {$1$};
        \node[below,font=\footnotesize] at (v1) {$-1$};
        \node[below,font=\footnotesize] at (v2) {$1$};
        \node[font=\small] at (1,2) {$\tilde f_2(x_2)$};
      \end{tikzpicture}
    };
    \node (z12) at (3,0)
    {
      \begin{tikzpicture}
        \coordinate (v0) at (0,1);
        \coordinate (v1) at (1,0);
        \coordinate (v2) at (2,0);
        \coordinate (w0) at (0,1.5);
        \coordinate (w2) at (2,1.5);
        \fill[blue!20,draw=black,thick] (w0) -- (v0) -- (v1) -- (v2) -- (w2);
        \fill (v0) circle (2pt);
        \fill (v1) circle (2pt);
        \fill (v2) circle (2pt);
        \node[below,font=\footnotesize] at (v0) {$1$};
        \node[below,font=\footnotesize] at (v1) {$-1$};
        \node[below,font=\footnotesize] at (v2) {$1$};
        \node[font=\small] at (1,2) {$\tilde f_2(tx_2)$};
      \end{tikzpicture}
    };
    \node (z11Prime) at (6,0)
    {
      \begin{tikzpicture}
        \coordinate (v0) at (0,0);
        \coordinate (v1) at (1,0);
        \coordinate (v2) at (2,1);
        \coordinate (w0) at (0,1.5);
        \coordinate (w2) at (2,1.5);
        \fill[blue!20,draw=black,thick] (w0) -- (v0) -- (v1) -- (v2) -- (w2);
        \fill (v0) circle (2pt);
        \fill (v1) circle (2pt);
        \fill (v2) circle (2pt);
        \node[below,font=\footnotesize] at (v0) {$-u_1$};
        \node[below,font=\footnotesize] at (v1) {$1$};
        \node[below,font=\footnotesize] at (v2) {$1$};
        \node[font=\small] at (1,2) {$\tilde f_2(1+t+t^2x_2)$};
      \end{tikzpicture}
    };
    \node (z12Prime) at (9,0)
    {
      \begin{tikzpicture}
        \coordinate (v0) at (0,0);
        \coordinate (v1) at (1,0);
        \coordinate (v2) at (2,1);
        \coordinate (w0) at (0,1.5);
        \coordinate (w2) at (2,1.5);
        \fill[blue!20,draw=black,thick] (w0) -- (v0) -- (v1) -- (v2) -- (w2);
        \fill (v0) circle (2pt);
        \fill (v1) circle (2pt);
        \fill (v2) circle (2pt);
        \node[below,font=\footnotesize] at (v0) {$u_1$};
        \node[below,font=\footnotesize] at (v1) {$-1$};
        \node[below,font=\footnotesize] at (v2) {$1$};
        \node[font=\small] at (1,2) {$\tilde f_2(t+t^2x_2)$};
      \end{tikzpicture}
    };
  \end{tikzpicture}\vspace{-3mm}
  \caption{The $\Delta(\tilde f_i(\tilde z_i)|_{u_i=x_i})$ and $\initial(\tilde a_j)$ from \cref{ex:approximateRootsMaximalPrecision}.}
  \label{fig:approximateRootsMaximalPrecision}
\end{figure}

\subsection{Roots of univariate polynomials}\label{sec:rootApproximations}

In this section, we discuss how to approximate the roots over the uncertainty ring.
The algorithm is a straightforward adaptation of the Newton-Puiseux expansion.

\begin{algorithm}[\texttt{puiseux\_expansion}]\label{alg:localfieldexpansion}\
  \begin{algorithmic}[1]
    \REQUIRE{$(\tilde f,w,b)$, where
      \begin{enumerate}[leftmargin=*]
      \item $\tilde f\in K[u_1,\dots,u_{i-1}][x_i]$ with unique Newton polygon,
      \item $w\in \Trop(\tilde f)$ a desired tropical point,
      \item $p_{\text{rel}}\in \RR$ an upper relative precision limit.
      \end{enumerate}}
    \ENSURE{$\tilde V\subseteq K[u_1,\dots,u_i]$ all approximate roots of $\tilde f$ with valuation $w$ and either with maximal precision or with precision $p_{\text{rel}}$.}
    \STATE Set $h\coloneqq \operatorname{in}_{w}(\tilde f)\in \mathfrak K[u_1,\dots,u_{i-1}][x_i]$.
    \IF{$p_{\text{rel}}\leq 0$ \OR $h\notin \mathfrak K[x_i]$}
    \RETURN{$\{u_i\cdot t^w\}$}
    \ENDIF
    \STATE Compute $V(h)\cap\mathfrak K^\ast$.
    \STATE Initialize $\tilde V\coloneqq\emptyset$.
    \FOR{$c\in V(h)\cap \mathfrak K^\ast$}
    \STATE Construct $\tilde f_{ct^w}\coloneqq \tilde f(x_i+c\cdot t^w)\in K[u_1,\dots,u_{i-1}][x_i]$.
    \IF{$\Delta(\tilde f_{ct^w})$ not unique.}
    \RETURN{$\{u_i\cdot t^w\}$}
    \ENDIF
    \IF{$\tilde f_{ct^w}=0$}
    \STATE Update $\tilde V\coloneqq \tilde V\cup \{c\cdot t^w\}$.
    \ELSE
    \FOR{$w'\in \Trop(\tilde f_{ct^w})$ with $w'>w$}
    \STATE Update
    \begin{equation*}
      \tilde V\coloneqq \tilde V\cup \Big\{ c\cdot t^w + z  \bigmid z\in \texttt{puiseux\_expansion}(\tilde f_{ct^w}, w', p_{\text{rel}}-(w'-w)) \Big\}
    \end{equation*}
    \ENDFOR
    \ENDIF
    \ENDFOR
  \end{algorithmic}
\end{algorithm}

\begin{example}\label{ex:localfieldexpansion}
  Consider $w = 0$, $p_{\text{rel}}=2$, and the polynomial
  $$\tilde f=(x_1-1-t^2)(x_1-1-t-t^2)=x_1^2 + (-2t^2 - t - 2)x_1 + (t^4 + t^3 + 2t^2 + t + 1).$$
  Note that the Newton polygon of $\tilde f$ is indeed unique and $w\in \Trop(\tilde f)$.  Hence our objects satisfy the requirements for \cref{alg:localfieldexpansion}.

  The first iteration does the following:
  \begin{itemize}
  \item Set $h \coloneqq \initial_w(\tilde f) = x_1^2-2x_1+1$.
  \item Since $p_{\text{rel}}>0$ and $h\in \mathfrak K[x]$, we continue computing $V(h)\cap \mathfrak K^\ast = \{1\}$ (which gives us $1\cdot t^0$ as the first term of our roots).
  \item For $c=1$, we construct
    $$\tilde f_{1\cdot t^0} \coloneqq \tilde f(x+1\cdot t^0)=(x_1-t^2)(x_1-t-t^2)=x_1^2 + (-2t^2 - t)x_1 + (t^4 + t^3).$$
  \item As $\tilde f_{1\cdot t^0}\neq 0$ and $\Delta(\tilde f_{1\cdot t^0})$ is unique, we continue with $w'\in \Trop(\tilde f_{1\cdot t^0}) = \{1, 2\}$, noting that both points are greater than $w$.
  \end{itemize}
  From here, the recursion with updated precision $p_{\text{rel}}'=p_{\text{rel}}-1=1$ branches: 

  For $w'=2$, we get $h' \coloneqq \initial_2(\tilde f_{1\cdot t^0}) = -x+1$ and $V(h') = \{1\}$, i.e., $1\cdot t^2$ as the next term of our root. For $c'=1$, we construct
  $$ (\tilde f_{1\cdot t^0})_{1\cdot t^2}= \tilde f(x_1+1+1\cdot t^2+x)=0.$$
  This means we have computed our root exactly, giving us the root $1+t^2$.

  For $w'=1$, we get $h' \coloneqq \initial_1(f_{1\cdot t^0}) = x_1^2-x_1$ and $V(h)\cap\mathfrak K = \{1\}$, i.e., $1\cdot t$ as the next term of our root.  For $c'=1$, we construct
  $$ (\tilde f_{1\cdot t^0})_{1\cdot t^1}= \tilde f(x_1+1+1\cdot t^1)=(x_1+t-t^2)(x_1-t^2)=x_1^2 + (-2t^2 + t)x_1 + (t^4 - t^3). $$
  This yields the tropicalization $\Trop(\tilde f_{1+1\cdot t^1}) = \{1, 2\}$ of which only $w''=2>1=w'$.  Note however that the new updated precision is $p_{\text{rel}}''=p_{\text{rel}}'-1=0$, so the recursion terminates in the next step, giving us the approximate root $1+t+t^2u_1$.

  The collective result is the set $\{1+t^2, 1+t+t^2u_1\}$.
\end{example}

\section{Tropicalizing triangular sets}\label{sec:tropicalizingTriangularSets}
In this section, we discuss how to tropicalize varieties of ideals generated by a triangular sets as in \cite[Section 2]{HofmannRen2018}.  We improve on \cite[Algorithm 2.10]{HofmannRen2018} through finer control on the precision during the tropicalization.  

Our algorithm tropicalizes $I$ by solving the system $f_1=\dots=f_n=0$ through backsubstitution.  The precision to which a root $z_1\in V(f_1)$ is computed governs the maximal possible precision to which a root $z_2\in V(f_2(z_1,x_2))$ can be computed, which, in turn, governs the maximal possible precision to which a root $z_3\in V(f_3(z_1,z_2,x_3))$ can be computed, and so on.  Roots $z_2,z_3,\dots$ are computed up to the maximal possible precision, and the precision of $z_1$ is kept minimal.

\subsection{Root trees of triangular sets}\label{sec:rootTree}
As mentioned in the beginning of the section, our algorithm works using simple univariate root approximation and backsubstitution.  We track the global state of the computation in a so-called root tree.  Let us first recall some basic concepts from graph theory:

\begin{definition}
  \label{def:tree}
  A \emph{(rooted) tree} is a finite simple graph $\Gamma$ with no cycles and a distinguished vertex referred to as the \emph{root}, which we usually denote $v_0$.  The degree of a vertex is the number of edges connecting to it, and vertices of degree $1$ are referred to as \emph{leaves}.

  Using $V=V(\Gamma)$, $E=E(\Gamma)$, and $B=B(\Gamma)$ to denote the vertices, edges, and branches of $\Gamma$, respectively.  Here, a \emph{branch} is a sequence of vertices $(v_0,\dots,v_k)$ such that $v_0\in V$ is the root and $(v_{i-1},v_i)\in E$ are edges for $i=1,\dots,k$.  We refer to $k$ as the \emph{length} of the branch $(v_0,\dots,v_k)$, and the \emph{depth} of vertex $v_k$.  The latter is justified as for any $v_k\in V$ there is a unique branch $(v_0,\dots,v_k)\in B$.  The depth of $\Gamma$ is the maximal depth of its vertices, which is necessarily attained at its leaves.
\end{definition}

\begin{definition}
  \label{def:rootTree}
  A \emph{root tree} of the triangular set $F$ is a rooted tree $\Gamma$ of depth at most $n$ together with two maps
  \begin{align*}
    \tilde z&\colon V(\Gamma)\setminus\{v_0\}\rightarrow K[u],\qquad v\mapsto \tilde z(v)\eqqcolon \tilde z_v,\\
    p&\colon V(\Gamma)\setminus\{v_0\}\rightarrow \RR_{\geq 0},\qquad \hspace{1.5mm}v\mapsto p(v)\eqqcolon p_v,
  \end{align*}
  such that, for any branch $(v_0,v_1,\dots,v_k)\in B(\Gamma)$, the images $\tilde z_{v_i}\in K[u_1,\dots,u_i]$ are approximate roots of $f_i(\tilde z_{v_1},\dots,\tilde z_{v_{i-1}},x_i)\in K[u_1,\dots,u_{i-1}][x_i]$ for $i=1,\dots,k$.

  The purpose of $p_{v_1}$ is to record the relative precision of $\tilde z_{v_1}$.  The purpose of $p_{v_i}$ for $i>1$ is to record the relative precision of $\tilde z_{v_1}$ used in the computation of $\tilde z_{v_i}$.
\end{definition}

\begin{example}
  \label{ex:rootTree}
  The following are two extremal examples of root trees:
  \begin{enumerate}
  \item \label{enumitem:rootTreeStart}
    Let $\Gamma$ consist of a single root $v_0$, and $\tilde z, p$ be empty mappings.  Then $(\Gamma, \tilde z, p)$ is trivially a root tree.
  \item \label{enumitem:rootTreeEnd}
    Consider $F\subseteq \CC\{\!\{t\}\!\}[x_1,x_2,x_3]$ from \cref{ex:tropicalVariety}.  A root tree $\Gamma$ consists of
    \begin{itemize}
    \item the vertices $v_0,v_{1,1},v_{1,2},v_{2,1},\dots,v_{2,4},v_{3,1},\dots,v_{3,4}$,
    \item the edges $(v_0,v_{1,1})$, $(v_0,v_{1,2})$, $(v_{1,1},v_{2,1})$, $(v_{1,1},v_{2,2})$, $(v_{1,2},v_{2,3})$, $(v_{1,2},v_{2,4})$, and $(v_{2,r},v_{3,r})$ for $r=1,\dots, 4$,
    \item the approximate roots $z_{1,1}= u_1$, $z_{1,2}= t^{-1} u_1$, $z_{2,1}=u_2$, $z_{2,2}=t^{-1}u_2$, $z_{2,3}=tu_2$, $z_{2,4}=t^{-1}u_2$, $z_{3,1}=u_3$, $z_{3,2}=tu_3$, $z_{3,3}=u_3$, $z_{3,4}=t^2u_3$,
    \item and all relative precisions $p_v=0$.
    \end{itemize}
    Note that $\Gamma$ is roughly dual to the tree in \cref{fig:seriesOfNewtonPolygons}.
  \end{enumerate}
\end{example}

The main goal of our algorithm will be how to transform the tree in \eqref{enumitem:rootTreeStart} to a tree as described  in \eqref{enumitem:rootTreeEnd}.  This motivates the next definition.

\begin{definition}
  \label{def:rootTreeInitialAndComplete}
  We refer to the trivial root tree in \cref{ex:rootTree} \eqref{enumitem:rootTreeStart} as the \emph{starting root tree}.

  Furthermore, we say a root tree $(\Gamma,\tilde z)$ is \emph{(tropically) complete}, if for every tropical point $(w_1,\dots,w_n)\in \Trop(\langle F\rangle)$ there is a branch $(v_0,v_1,\dots,v_n)\in B(\Gamma)$ such that $(\val(\tilde z_{v_1}),\dots,\val(\tilde z_{v_n}))=(w_1,\dots,w_n)$. 
\end{definition}

\subsection{The tropicalization algorithm}\label{sec:algorithm}

In this section, we detail how the tropicalization algorithm in \texttt{OscarZerodimensionalTropicalization.jl} works.

\begin{definition}
  \label{def:extensionAndReinforcementPolynomials}
  Let $(\Gamma,\tilde z,p)$ be a root tree, let $v_k\in V(\Gamma)$ be a vertex and let $(v_0,\dots,v_k)\in B(\Gamma)$ the branch ending at $v_k$ of length $k<n$.

  The \emph{extension polynomial} past $v_k$ is the univariate polynomial
  \begin{equation*}
    f_{k+1}(\tilde z_{v_1},\dots,\tilde z_{v_k},x_{k+1})\in K[u_1,\dots,u_k][x_{k+1}].
  \end{equation*}

  The \emph{reinforcement polynomial} at $v_k$ is the univariate polynomial
  \begin{equation*}
    f_{k}(\tilde z_{v_1},\dots, \tilde z_{v_{k-1}}, x_k+\tilde z_{v_k})\in K[u_1,\dots,u_k][x_k].
  \end{equation*}
\end{definition}

We will be using extension polynomials to extend a root tree beyond a leaf $v_k$ and reinforcement polynomials to improve the precisions of the approximate root $\tilde z_{v_k}$.

\begin{algorithm}[\texttt{grow!}]\
  \label{alg:extension}
  \begin{algorithmic}[1]
    \REQUIRE{$(\Gamma,v_k)$, where
      \begin{enumerate}[leftmargin=*]
      \item $\Gamma$ a root tree, and
      \item $v_k\in V(\Gamma)$ a leaf whose extension polynomial has a unique Newton polygon.
      \end{enumerate}}
    \ENSURE{\texttt{nothing}, extends $\Gamma$ so that $v_k$ is no leaf anymore.}
    \STATE Let $(v_0,\dots,v_k)\in B(\Gamma)$ be the branch ending at $v_k$.
    \FOR{$w\in \Trop(f_{k+1}(\tilde z_{v_1},\dots,\tilde z_{v_k},x_{k+1}))$}
    \STATE Add a new vertex $v'$ and a new edge $\{v_k,v'\}$ to $\Gamma'$.
    \STATE Set $\tilde z_{v'}\coloneqq u_{k+1}\cdot t^w$ and $p_{v'}\coloneqq 0$.
    \ENDFOR
  \end{algorithmic}
\end{algorithm}

Assuming there is a leaf $v_k$ whose extension polynomial $f_{k+1}(\tilde z_{v_1},\dots,\tilde z_{v_k},x_{k+1})$ has no unique Newton polygon, there are many ways how to proceed to improve the precision of $\tilde z_{v_1}, \dots, \tilde z_{v_k}$.  The following algorithm describes how it is done in our implementation:

\begin{algorithm}[\texttt{reinforce!}]\
  \label{alg:reinforcement}
  \begin{algorithmic}[1]
    \REQUIRE{$(\Gamma,v_k,p_{\text{step}},p_{\text{max}})$, where
      \begin{enumerate}[leftmargin=*]
      \item $\Gamma$ a root tree,
      \item $v_k\in V(\Gamma)$,
      \item $p_{\text{step}}\in \RR_{>0}$, a value by which to increment the precision of $\tilde z_{v_1}$,
      \item $p_{\text{max}} \in \RR_{\geq 0}$, a upper bound for the relative precisions as safeguard.
      \end{enumerate}}
    \ENSURE{\texttt{nothing}: Improves precision of the $\tilde z$ along the branch ending at $v_k$. This may add new vertices and edges to $\Gamma$. 
    }
    \STATE Let $(v_0,\dots,v_k)\in B(\Gamma)$ be the branch ending at $v_k$.
    \STATE Let $l=\min(\{1\}\cup \{j\mid j=2,\dots,k \text{ with } p_{v_j}<p_{v_1}\})$ be the index of the first branch vertex whose root precision needs to be increased.
    \STATE Construct the reinforcement polynomial $\tilde f\coloneqq f_l(\tilde z_{v_1},\dots,\tilde z_{v_{l-1}},\tilde z_{v_{l}}+x_l)$.
    \IF{$l=1$}
    \STATE In $\Gamma$, update the precision $p_{v_1}\coloneqq p_{v_1}+p_{\text{step}}$.
    \STATE Set $p_{\text{rel}}\coloneqq p_{v_1}$
    \ELSE
    \STATE In $\Gamma$, update the precision $p_{v_l}\coloneqq p_{v_1}$.
    \STATE Set $p_{\text{rel}}\coloneqq p_{\text{max}}$
    \ENDIF
    \STATE Compute $\tilde V\coloneqq \{ \tilde z_{v_l}+\tilde z'\mid z'\in \texttt{puiseux\_expansion}(\tilde f, 0, p_{\text{rel}})\}$.
    \STATE Let $\Gamma_{v_l}\subseteq\Gamma$ be the subtree below $v_l$ and delete $\Gamma_{v_l}$ from $\Gamma$.
    \FOR{$\tilde z_{v_l}'\in \tilde V$}
    \STATE Copy $\Gamma'\coloneqq \Gamma_{v_l}$.
    \STATE In $\Gamma'$, update the root $\tilde z_{v_l} \coloneqq \tilde z_{v_l} + \tilde z_{v_l}'$.
    \STATE Attach $\Gamma'$ to $v_{l-1}\in\Gamma$.
    \ENDFOR
  \end{algorithmic}
\end{algorithm}

Note that \cref{alg:reinforcement} may need to be called multiple times to improve the precision $v_k$, and it may also split branches, see \cref{ex:reinforcementPolynomial}.

Our main algorithm then is as follows:
\begin{algorithm}[\texttt{trop\_triangular}]\label{alg:tropTriangular}\
  \begin{algorithmic}[1]
    \REQUIRE{$(F,p_{\text{step}},p_{\text{max}})$, where
      \begin{enumerate}[leftmargin=*]
      \item $F=\{f_1,\dots,f_n\}\subseteq K[x]$ a triangular set,
      \item $p_{\text{step}}\in\RR_{>0}$, a value by which to increment the precision,
      \item $p_{\text{max}}\in\RR_{\geq 0}$, an upper bound for the relative precisions as safeguard.
      \end{enumerate}}
    \ENSURE{$\Trop(\langle F\rangle)\subseteq\RR^n$, the tropicalization of $F$.}
    \STATE Let $\Gamma$ be the starting root tree.
    \WHILE{there is a leaf $v_k\in V(\Gamma)$ of depth $k<n$}
    \STATE Let $(v_0,v_1,\dots,v_k)\in B(\Gamma)$ be the branch ending at $v_k$.
    \STATE Let $\tilde f\coloneqq f_{k+1}(\tilde z_{v_1},\dots,\tilde z_{v_k},x_{k+1})\in K[u_1,\dots,u_k][x_{k+1}]$ be the extension polynomial at $v_k$.
    \IF{the Newton polygon $\Delta(\tilde f)$ is unique}
    \STATE $\texttt{grow!}(\Gamma,v_k)$.
    \ELSE
    \STATE $\texttt{reinforce!}(\Gamma,v_k,p_{\text{step}},p_{\text{max}})$.
    \ENDIF
    \ENDWHILE
    \STATE Read off $\Trop(\langle F\rangle)$ from $\Gamma$:
    \begin{equation*}\vspace{-1em}
        \Trop(\langle F\rangle)\coloneqq \Big\{ (\val(\tilde z_{v_1}),\dots,\val(\tilde z_{v_n})) \bigmid (v_1,\dots,v_n)\in B(\Gamma) \Big\}.
    \end{equation*}
    \RETURN{$\Trop(\langle F\rangle)$}
  \end{algorithmic}
\end{algorithm}

\begin{example}
  \label{ex:reinforcementPolynomial}
  Consider the triangular system $F=\{f_1,f_2\}$ given by
  \begin{align*}
      f_1 = (x_1-1-t^2)(x_1-1-t-t^2)\quad\text{and}\quad f_2 = x_2-(x_1-1-t),
  \end{align*}
  as well as the precisions $p_{\text{step}}=p_{\text{max}}=2$.
  The polynomial $f_1$ is taken from \cref{ex:localfieldexpansion} and it is straightforward to see that $\Trop(\langle F\rangle)=\{(0,1), (0,2)\}$. We run \cref{alg:tropTriangular} on this example, each step representing when the root tree changes, see \cref{fig: main example}:

  Step 1: The main algorithm starts with a trivial root tree $\Gamma$, whose only leaf is the vertex $v_0$. As $\tilde f=f_1$, the empty Newton polygon $\Delta(\tilde f)$ is unique (with slope $0$), hence the algorithm calls \texttt{grow!} at $v_0$.
  
  Step 2: This adds a new vertex $v_1$, and extends the maps $\tilde z$ and $p$ by setting $\tilde z(v_1) = u_1$ and $\tilde p(v_1) = 0$.  Now the only leaf of $\Gamma$ is $v_1$ for which $\tilde f=f_2(\tilde z(v_1),x_2)=x_2-(u_1-1-t)$ whose Newton polygon is not unique.  Hence the algorithm proceeds by calling \texttt{reinforce!} on $v_1$.
  
  Step 3: This involves \texttt{puiseux\_expansion} on $f_1$ at valuation $0$ up to relative precision $p_{\text{rel}}=p(v_1)+p_{\text{step}}=2$, which gives us two approximate roots $\{1+t^2, 1+t+t^2u_1\}$ from \cref{ex:localfieldexpansion}.  Hence $v_1$ is split into two vertices $v_{1,1}$ and $v_{1,2}$ with $\tilde z(v_{1,1})= 1+t^2$, $p(v_{1,1})=2$, $\tilde z(v_{1,2})= 1+t+t^2u_1$, $p(v_{1,2})=2$. All resulting Newton polygons are unique, hence the algorithm calls \texttt{grow!} on $v_{1,1}$ and $v_{1,2}$.

  Steps 4 and 5: This adds new vertices $v_{2,1}$ and $v_{2,2}$ with $\tilde z(v_{1,1})= tu_2$, $p(v_{2,1})=0$, $\tilde z(v_{2,2})=t^2u_2$, $p(v_{2,2})=0$.  This completes the root tree $\Gamma$ and yields $\Trop(\langle F\rangle)=\{(0,1), (0,2)\}$.

  Observe that due to $p_{\text{step}}=2$, it was sufficient to call \texttt{reinforce!} (and thus \texttt{puiseux\_expansion}) once in Step 3.  If $p_{\text{step}}$ were $1$, we would have needed to call \texttt{reinforce!} twice to obtain the required precision.
\end{example}

\begin{figure}
    \centering
    \begin{tabular}{c | c | c | c | c | c }
        \textbf{Step} & \textbf{1} & \textbf{2} & \textbf{3} & \textbf{4}  & \textbf{5} \\ \hline \hline
        Tree & \Lazypic{1.5cm}{\begin{tikzpicture}
    \node[draw, circle] at (0, 2)   (root) {$r$};\draw[opacity = 0] (0,0) circle (.4cm) ;
    \node[opacity = 0] at (-0, -0.75)    {$v$};
        \end{tikzpicture}} & \Lazypic{2cm}{\begin{tikzpicture}
    \node[draw, circle] at (0, 2)   (root) {$r$};
    \node[draw, circle] at (0, 1)   (v1) {$0$};
    \draw (root) -- (v1); \draw[opacity = 0] (0,0) circle (.4cm) ;
    \node[opacity = 0] at (-0, -0.75)    {$v_{x,1}$};
        \end{tikzpicture}}& \Lazypic{2cm}{\begin{tikzpicture}
    \node[draw, circle] at (0, 2)   (root) {$r$};
    \node[draw, circle] at (-0.5, 1)   (v1) {0};
    \node[draw, circle] at (0.5, 1)   (v2) {0};
    \draw (root) -- (v1);\draw (v2) -- (root);\draw[opacity = 0] (0,0) circle (.4cm) ;
    \node[opacity = 0] at (-0.5, -0.75)    {$v_{x,1}$};
        \end{tikzpicture}}&\Lazypic{2cm}{\begin{tikzpicture}
    \node[draw, circle] at (0, 2)   (root) {$r$};
    \node[draw, circle] at (-.5, 1)   (v1) {0};
    \node[draw, circle] at (.5, 1)   (v2) {0};
    \node[draw, circle] at (-.5,0) (v3) {1};
    \draw (root) -- (v1);\draw (v2) -- (root); \draw (v1) -- (v3);
    \node[opacity = 0] at (-0.5, -0.75)    {$v_{x,1}$};
        \end{tikzpicture}} &  \Lazypic{2.5cm}{\begin{tikzpicture}
    \node[draw, circle] at (0, 2)   (root) {$r$};
    \node[draw, circle] at (-.5, 1)   (v1) {0};
    \node[draw, circle] at (.5, 1)   (v2) {0};
    \node[draw, circle] at (-.5,0) (v3) {1};
    \node[draw, circle] at (.5,0) (v4) {2};
    \draw (root) -- (v1);\draw (v2) -- (root);\draw (v3) -- (v1); \draw (v2) -- (v4);\node[opacity = 0.9] at (-1.3, 2)    {\small $v_0$};\node[opacity = 0.9] at (-1.3, 1)    {\small $v_{1}$};\node[opacity = 0.9] at (-1.3, 0)    {\small $v_{2}$};
    \node[opacity = 0.9] at (-0.5, -0.75)    {\small $v_{x,1}$};\node[opacity = 0.9] at (0.5, -0.75)    {\small $v_{x,2}$};
        \end{tikzpicture}} \\ \hline
        Procedure& \texttt{grow!} &  \texttt{reinforce!} & \texttt{grow!}  & \texttt{grow!} & \texttt{done!}\\
    \end{tabular}
    \caption{Each column depicts the root tree at a different step of \cref{alg:tropTriangular} performed in Example \ref{ex:reinforcementPolynomial}. Inside each vertex $v$ is the valuation  of $\tilde z (v)$. The last row of the table indicates the next sub-procedure the program performs. In the last step, the labelling of the vertices is indicated. }
    \label{fig: main example}
\end{figure}

\renewcommand*{\bibfont}{\small}
\printbibliography
\end{document}